\documentclass[10pt]{amsart}

\usepackage{amssymb,amsthm,amstext,amsfonts,amsmath}

\newtheoremstyle{theorem}
     {11pt}
     {11pt}
     {}
     {}
     {\bfseries}
     {}
     {.5em}
     {\noindent\thmnumber{#2}. \thmname{#1}\thmnote{#3}}

\theoremstyle{theorem}

\newcommand{\arrow}{\mathbb{A}}
\newcommand{\csf}{{}\sp\omega{2}}
\newcommand{\C}{\mathfrak{c}}
\newcommand{\U}{\mathcal{U}}

\newcommand{\sorg}{\mathbb{S}}

\newcommand{\p}{\mathfrak{p}}
\newcommand{\Q}{\mathbb{Q}}
\newcommand{\cl}[2][X]{\mathrm{cl}_{#1}({#2})}
\newcommand{\B}{\mathcal{B}}
\newcommand{\D}{\mathcal{D}}
\newcommand{\pair}[1]{\langle #1 \rangle}

\newtheorem{thm}{Theorem}[section]
\newtheorem{lemma}[thm]{Lemma}
\newtheorem{propo}[thm]{Proposition}
\newtheorem{coro}[thm]{Corollary}

\newtheorem{exs}[thm]{Examples}

\newenvironment{ex2}{}{\hfill\qed\vspace{1ex}}

\title{Countable Dense Homogeneity and the Double Arrow Space}
\author{Rodrigo Hern\'andez-Guti\'errez}
\address{Centro de Ciencias Matem\'aticas, UNAM, A.P. 61-3, Xangari, Morelia, Michoac\'an, 58089, M\'exico}	
\email{rod@matmor.unam.mx}
\date{\today}


\begin{document}

\begin{abstract}
Let $\arrow$ denote the Alexandroff-Urysohn double arrow space. We prove the following results: (a) $\arrow\times\csf$ is not countable dense homogeneous; (b) ${}\sp{\omega}{\arrow}$ is not countable dense homogeneous; (c) $\arrow$ has exactly $\C$ types of countable dense subsets. These results answer questions by Arhangel'ski\u \i, Hru\v s\'ak and van Mill.
\end{abstract}

\maketitle

\section{Introduction}

All spaces under consideration are assumed to be Hausdorff. A separable space $X$ is \emph{countable dense homogeneous} (\emph{CDH} for short) if every time $D$ and $E$ are countable dense subsets of $X$ there exists a homeomorphism $h:X\to X$ such that $h[D]=E$. Sections 14, 15 and 16 of \cite{arh-vm-homogeneity} provide a good reference on countable dense homogeneity. Some examples of CDH spaces include the real line, the Cantor set, second-countable manifolds and the Hilbert cube.

If one looks carefully at these examples of well-known spaces that are CDH, the common property among them is that they are all completely metrizable. So it is natural to ask whether there are more exotic spaces that are CDH. For example, in \cite[Corollary 2.4]{cdhdefinable} it was proved that CDH Borel spaces are completely metrizable and in \cite{lambda} it was shown that there is a non-definable CDH subspace of the real line.

In another direction, there are no known examples of CDH compact spaces of uncountable weight in ZFC. By the results in \cite[Theorem 4]{steprans-zhou} and \cite[Theorem 3.3]{cdhdefinable} it is known that the Cantor cube ${}\sp{\kappa}2$ is CDH if and only if $\kappa<\p$, where $\p$ is the pseudointersection number (see \cite{vd62}). Also, in \cite[Section 4]{arh-vm-cdh-cardinality} it is shown that CH implies that there is a CDH compact space of uncountable weight.

Let us recall the definition of the classical Alexandroff-Urysohn double arrow space $\arrow$ first defined in \cite{alex-uryh-compact}. Let $\arrow_0=(0,1]\times\{0\}$, $\arrow_1=[0,1)\times\{1\}$ and $\arrow=\arrow_0\cup\arrow_1$. Define the lexicographic (strict) order on $\arrow$ as $\pair{x,t}<\pair{y,s}$ if $x<y$ or $x=y$ and $t<s$. Then $\arrow$ is given the order topology. It is known that $\arrow$ is separable, first-countable, compact, $0$-dimensional and of weight $\C$. Notice that both $\arrow_0$ and $\arrow_1$ have the Sorgenfrey line topology as subspaces of $\arrow$ and both are dense in $\arrow$. Also, the function $\pi:\arrow\to[0,1]$ defined by $\pi(\pair{x,t})=x$ is a $\leq 2$-to-$1$ continuous function.

In \cite{arh-vm-cdh-cardinality}, A. Arhangel'ski\u \i{} and J. van Mill proved that $\arrow$ is not CDH and asked whether $\csf\times\arrow$ is CDH. Also, M. Hru\v s\'ak and J. van Mill have asked whether ${}\sp\omega\arrow$ is CDH. The objective of this note is to answer these questions. Moreover, the results presented shed some information about restrictions to CDH spaces that are products. We also show that there are exactly $\C$ different types of countable dense subsets of $\arrow$.

Let us recall some standard definitions, for a reference on general topology see \cite{eng}. A space is \emph{crowded} if it is non-empty and has no isolated points. A \emph{$\pi$-base} in a space $X$ is a collection $\B$ of non-empty open sets of $X$ such that for every non-empty open set $U\subset X$ there exists $V\in\B$ such that $V\subset U$. The \emph{$\pi$-weight} of $X$ is defined as the minimal cardinality of a $\pi$-base of $X$. We call $X\subset\csf$ a \emph{Bernstein set} if both $X$ and $\csf\setminus X$ intersect every closed crowded subspace of $\csf$. All intervals in a linearly ordered space are assumed to be non-empty. A space $X$ is homogeneous if for every $x,y\in X$ there is a homeomorphism $h:X\to X$ such that $h(x)=y$.

\section{Products that are CDH}

In this section we will prove that neither $\arrow\times\csf$ nor ${}\sp{\omega}{\arrow}$ are CDH. First we will give a key property of $\arrow$, compare with Theorem 4.6 of \cite{justin-sorgenfrey} and the well-known fact that every closed and crowded subspace of $\csf$ is homeomorphic to $\csf$. If $(X,<)$ is a linearly ordered set, an interval $[a,b]\subset X$ with $a<b$ and $X\cap(a,b)=\emptyset$ will be called a \emph{jump}.

\begin{propo}\label{closedinarrow}
Every closed and crowded subset of $\arrow$ is homeomorphic to $\arrow$.
\end{propo}
\begin{proof}
Let $X\subset\arrow$ be closed and crowded. We will construct a homeomorphism $h:X\to\arrow$.

Since $X$ is compact, $T=\pi[X]$ is a compact subspace of $[0,1]$. Notice that since $X$ is crowded, $T$ is crowded as well. View $T$ as an ordered subspace of $[0,1]$ and consider the equivalence relation $\sim$ on $T$ obtained by defining $x\sim y$ if either $x=y$ or $[\min{\{x,y\}},\max{\{x,y\}}]$ is a jump in $T$. Since $T$ is crowded, each equivalence class of $\sim$ consists of at most two points. Since the equivalence classes are convex, the quotient $T/\sim$ can be linearly ordered in a natural way. Furthermore, it is easy to check that the order topology makes $T/\sim$ into a separable continuum. It follows that $T/\sim$ is order-isomorphic to $[0,1]$. Thus, there exists a continuous function $f:T\to[0,1]$ that is order preserving and the following property holds.

\begin{quote}
$(\ast)$ Fix $t\in[0,1]$. Then $|f\sp\leftarrow(t)|\neq 1$ if and only if $f\sp\leftarrow(t)=\{x,y\}$, where $0<x<y<1$ and $[x,y]$ is a jump in $T$. Moreover, if $f\sp\leftarrow(t)=\{x\}$, then $x$ is not in a jump of $T$.
\end{quote}

The function $f$ works just like the classical continuous function that maps the Cantor set onto $[0,1]$ by collapsing jumps (see \cite[Exercise 3.2.B]{eng}). The following is not hard to prove from the fact that $X$ is crowded.

\begin{quote}
$(\star)$ If $[x,y]$ is a jump in $T$, then $X\cap\{\pair{x,0},\pair{x,1}\}=\{\pair{x,0}\}$ and $X\cap\{\pair{y,0},\pair{y,1}\}=\{\pair{y,1}\}$.
\end{quote}

Define $h:X\to\arrow$ by $h(\pair{x,t})=\pair{f(x),t}$. We will prove that $h$ is a homeomorphism. Since both $X$ and $\arrow$ are linearly ordered topological spaces, it is enough to prove that $h$ is an order-isomorphism.

Let $\pair{x,t}<\pair{y,s}$ be two points in $X$. If $x=y$, then $0=t<s=1$ so $h(\pair{x,t})=\pair{f(x),t}<\pair{f(x),s}=h(\pair{y,s})$. If $f(x)<f(y)$, then also $h(\pair{x,t})<h(\pair{y,s})$. So assume that $x<y$ and $f(x)=f(y)$. By $(\ast)$, $[x,y]$ is a jump in $T$. Using $(\star)$ it is not hard to see that $t=0$ and $s=1$. Thus, $h(\pair{x,t})=\pair{f(x),0}<\pair{f(x),1}=h(\pair{y,s})$.

Finally, we show that $h$ is onto. Let $\pair{x,t}\in\arrow$. By $(\ast)$, there are two cases.

\vskip12pt
\noindent Case 1: There is a unique $y\in T$ with $x=f(y)$.
\vskip12pt

By $(\ast)$, $y$ is not in a jump. If $y\in\{0,1\}$, then clearly $\pair{y,1-y}\in T$ and $h(\pair{y,1-y})=\pair{x,t}$. Assume that $y\notin\{0,1\}$. Then for all $n<\omega$ there are $z_n\sp0,z_n\sp1\in T$ such that 
$$
y-\frac{1}{n+1}<z_n\sp0<y<z_n\sp1<y+\frac{1}{n+1}.
$$
For each $n<\omega$ and $i\in 2$, let $t_n\sp i\in 2$ be such that $\pair{z_n\sp i,t_n\sp i}\in X$. Notice that $\lim_{n\to\infty}{\pair{z_n\sp0,t_n\sp0}}=\pair{y,0}$ and $\lim_{n\to\infty}{\pair{z_n\sp0,t_n\sp0}}=\pair{y,1}$. By the compactness of $X$, $\{\pair{y,0},\pair{y,1}\}\subset X$. Thus, $\pair{y,t}\in X$ and $h(\pair{y,t})=\pair{x,t}$.

\vskip12pt
\noindent Case 2: There are $y_0,y_1\in T$ with $y_0<y_1$ and $x=f(y_0)=f(y_1)$.
\vskip12pt

By $(\ast)$, $[y_0,y_1]$ is a jump in $T$. By $(\star)$, we obtain that $\pair{y_i,i}\in X$ for $i\in2$. Thus, $\pair{y_t,t}\in X$ and $h(\pair{y_t,t})=\pair{x,t}$.

This concludes the proof that $h$ is the homeomorphism that we were looking for.
\end{proof}

\begin{coro}
Any compact and metrizable subset of $\arrow$ is scattered, thus, countable. In particular, $\arrow$ does not contain topological copies of $\csf$.
\end{coro}

The following result is a general theorem that places restrictions on CDH spaces that are products. Answers to the questions on $\arrow$ will follow from it.

\begin{thm}\label{main}
Let $X$ and $Y$ be two crowded spaces of countable $\pi$-weight. If $X\times Y$ is CDH, then $X$ contains a subset homeomorphic to $\csf$ if and only if $Y$ contains a subset homeomorphic to $\csf$.
\end{thm}
\begin{proof}
Assume that $X$ contains a subspace homeomorphic to the Cantor set and $Y$ does not, we shall arrive to a contradiction. Since $X\times Y$ contains a Cantor set, there is a countable dense subset $D\subset X\times Y$ and $Q\subset D$ such that $\cl[X\times Y]{Q}\approx\csf$. We shall construct a countable dense subset $E\subset X\times Y$ that does not have this property.

Let $\B=\{U_n\times V_n:n<\omega\}$ be a $\pi$-base of the product $X\times Y$, where $U_n$ is open in $X$ and $V_n$ is open in $Y$ for each $n<\omega$. Let $\pi_X:X\times Y\to X$ and $\pi_Y:X\times Y\to Y$ be the projections. Recursively, choose $\{e_n:n<\omega\}\subset X\times Y$ such that $\pi_X(e_n)\in U_n$ and
$$
\pi_Y(e_n)\in V_n\setminus\{\pi_Y(e_0),\ldots,\pi_Y(e_{n-1})\}
$$
for all $n<\omega$. Let $E=\{e_n:n<\omega\}$. Thus, $\pi_Y\!\!\!\restriction_{E}:E\to Y$ is one-to-one.

Assume that there is some autohomeomorphism of $X\times Y$ that takes $D$ to $E$. Then there exists $R\subset E$ such that $K=\cl[X\times Y]{R}$ is homeomorphic to $\csf$. Notice that $T=\pi_Y[K]$ is a compact subset of $Y$ of countable weight. Since $Y$ does not contain topological copies of the Cantor set, $T$ is scattered. Then $T$ contains an isolated point $p$. Since $(X\times\{p\})\cap K$ is a clopen subset of $K$, $\pi_Y\restriction_R:R\to Y$ is one to one and $R$ is dense in $K$, we obtain that $(X\times\{p\})\cap K$ is a singleton. But $K$ is crowded so this is a contradiction. Thus, the theorem follows.
\end{proof}

We immediately obtain the following. This answers Question 3.3 of \cite{arh-vm-cdh-cardinality}.

\begin{coro}
$\arrow\times\csf$ is not CDH.
\end{coro}

The following corollary gives a necessary condition on $X$ for ${}\sp\omega X$ to be CDH. It is the first criterion of this kind that works for non-metrizable spaces. See also the discussion in Example \ref{example}(c). This contrasts with the fact that if $X$ is first-countable and $0$-dimensional, then ${}\sp\omega X$ is homogeneous (see \cite{dow-pearl}).

\begin{coro}\label{coropower}
Let $Z$ be a crowded space of countable $\pi$-weight. If ${}\sp\omega Z$ is CDH, then $Z$ contains a subspace homeomorphic to $\csf$.
\end{coro}
\begin{proof}
Assume that $Z$ contains no subspace homeomorphic to $\csf$. Let $X=Z$ and $Y={}\sp\omega Z$. Then both $X$ and $Y$ are crowded spaces of countable $\pi$-weight. Notice that $Z$ has at least two points so $Y$ contains a topological copy of $\csf$. By Theorem \ref{main}, we obtain a contradiction.
\end{proof}

\begin{exs}\label{example}
$\empty$
\end{exs}
\begin{ex2}
\noindent\!\! $(a)$ If $\Q$ is the space of rational numbers, ${}\sp\omega\Q$ is not CDH, this was first shown by Fitzpatrick and Zhou (\cite[Corollary to Theorem 3.5]{fitz-zhou-cdhbaire}).\vskip12pt
\noindent\!\! $(b)$ If $\sorg$ is the Sorgenfrey line, using back-and-forth it is easy to prove that $\sorg$ is CDH. However, ${}\sp\omega\sorg$ is not CDH.\vskip12pt
\noindent\!\! $(c)$ It is an open question whether there exists a ZFC example of a $0$-dimensional separable metrizable space $X$ such that ${}\sp\omega X$ is CDH while $X$ is not completely metrizable (see \cite[Question 3.2]{cdhdefinable}). In \cite[Theorem 3.1]{cdhdefinable} it was shown that such an $X$ must be a Baire space and not Borel. A natural candidate is a Bernstein set, as it is known that MA implies the existence of a CDH Bernstein set (\cite[Theorem 3.5]{CDH_MA}). However, by Corollary \ref{coropower}, such an $X$ cannot be a Bernstein set. In fact, the only (consistently) known non-Borel spaces $X$ such that ${}\sp\omega X$ is CDH are filters (see \cite{medinimilovich} and \cite{cdhnonmeagerPfilter}), which do contain topological copies of the Cantor set.\vskip12pt
\noindent\!\! $(d)$ Finally, ${}\sp\omega\arrow$ is not CDH, as announced.
\end{ex2}

Notice that Theorem \ref{main} applies to first-countable separable spaces. This observation is important because, by \cite[Corollary 2.5]{arh-vm-cdh-cardinality}, it is consistent that all compact CDH spaces are first-countable.

\section{Number of dense subsets}

Recently, in \cite{hrusak-vm-cdh}, M. Hru\v s\'ak and J. van Mill have studied the topological types of countable dense subsets of separable metrizable spaces. If $X$ is a separable space and $\kappa$ a cardinal number, we will say that $X$ has $\kappa$ \emph{types of countable dense subsets} if $\kappa$ is the minimum cardinal number such that there is a collection $\D$ of countable dense subsets of $X$ such that $|\D|=\kappa$ and for any countable dense subset $E\subset X$ there is $D\in\D$ and a homeomorphism $h:X\to X$ such that $h[D]=E$.

In \cite{hrusak-vm-cdh} it is proved, among other results, that a locally compact separable metrizable space that is homogeneous and not CDH has uncountably many types of countable dense subsets. In this section, we will prove that $\arrow$ has $\C$ types of countable dense subsets. 

Autohomeomorphisms of $\arrow$ behave in a very simple manner, as the following shows. We do not include the proof since it follows directly from the arguments given in Lemma 3.1 and Theorem 3.2 from \cite{arh-vm-cdh-cardinality}.
 
\begin{propo}\label{homeomorphisms}
Let $h:\arrow\to\arrow$ be a homeomorphism. Then there exists a collection $\U$ of pairwise disjoint clopen intervals of $\arrow$ such that $\bigcup\U$ is dense in $\arrow$ and for every $J\in\U$, $h\restriction_{J}:J\to\arrow$ is either increasing or decreasing.
\end{propo}

We also need the following result. The first proof of this fact was apparently given by Stefan Mazurkiewicz and Wac{\l}aw Sierpi\'nski in \cite{mazurkiewicz-sierpinski}. A different proof was given by Brian, van Mill and Suabedissen in \cite[Lemma 14]{brian-vanmill-suabedissen}. 

\begin{lemma}\label{countablesubsets}
The number of distinct homeomorphism classes of countable metrizable spaces is $\C$.
\end{lemma}

In \cite[Theorem 3.2]{arh-vm-cdh-cardinality} two different types of countable dense subsets of $\arrow$ were found. We will manipulate those two dense subsets and get $\C$ different ones.

\begin{thm}\label{arrowtypes}
$\arrow$ has $\C$ types of countable dense subsets.
\end{thm}
\begin{proof}
Since $|\arrow|=\C$, it is enough to find $\C$ countable dense subsets of $\arrow$ that are different with respect to autohomeomorphisms of $\arrow$. 

The classical middle-thirds Cantor set in $[0,1]$ is the complement of a union of countably many open intervals $\{(x_n,y_n):n<\omega\}$. For each $n<\omega$, let $J_n=[\pair{x_n,1},\pair{y_n,0}]$, this is a clopen subinterval of $\arrow$. Define 
$$
X=\arrow\setminus\bigcup\{J_n:n<\omega\}.
$$

Then $X$ is a closed, crowded and nowhere dense subset of $\arrow$. Since $X$ is also first-countable and regular, every countable dense subset of $X$ is metrizable and crowded, hence homeomorphic to the space $\mathbb{Q}$ of rational numbers. Since $\mathbb{Q}$ is universal for countable metrizable spaces (see \cite[Exercise 4.3.H(b)]{eng}), by Lemma \ref{countablesubsets} there is a collection $\{C_\alpha:\alpha<\C\}$ of countable subspaces of $X$ that are pairwise non-homeomorphic.

Let $Q_0$, $Q_1$ be two disjoint countable dense subsets of $\bigcup\{(x_n,y_n):n<\omega\}$ (with the Euclidean topology). For each $n<\omega$, choose $z_n\in (x_n,y_n)$ and define the following sets: $K_n\sp0=(x_n,z_n)$, $K_n\sp1=(z_n,y_n)$, $J_n\sp0=[\pair{x_n,1},\pair{z_n,0}]$ and $J_n\sp1=[\pair{z_n,1},\pair{y_n,0}]$. Let $D$ be the subset of $\arrow\setminus X$ defined such that

\begin{quote}
$(\ast)$ if $n<\omega$ and $i,j\in 2$, then $\pi[J_n\sp{i}\cap\arrow_j\cap D]=Q_{i\cdot j}\cap K_n\sp{j}$.
\end{quote}

Notice that $D$ is then a countable dense subset of $\arrow\setminus X$ and thus, of $\arrow$. For each $\alpha<\C$, let $D_\alpha=D\cup C_\alpha$. Then $D_\alpha$ is a countable dense subset of $\arrow$. Assume that there are $\beta<\gamma<\C$ and a homeomorphism $h:\arrow\to\arrow$ such that $h[D_\beta]=D_\gamma$, we shall arrive to a contradiction.

\vskip6pt
\noindent Claim: $h[X]=X$.
\vskip6pt

We shall prove the claim proceeding by contradiction. Without loss of generality, we can assume that there are $k<\omega$, $i\in 2$ and $x\in J_k\sp{i}$ such that $h(x)\in X$. By the definition of the middle-thirds Cantor set, it is easy to see that every neighborhood of a point of $X$ contains some interval from $\{J_n:n<\omega\}$. Thus, by continuity there exist $m<\omega$ and clopen intervals $I_0$ and $I_1$ such that $I_0\subset J_k\sp{i}$, $I_1\subset J_m\sp{1-i}$ and $h[I_0]\subset I_1$.

By Proposition \ref{homeomorphisms}, we may assume that $h$ is increasing or decreasing on $I_0$. Let us consider the case when $h$ is increasing on $I_0$, the other case is treated similarly. Compactness implies that every clopen subset of $\arrow$ is a finite union of clopen intervals. Thus, $h[I_0]$ can be written as the union of finitely many clopen intervals $[a_0,b_0]\cup\ldots\cup[a_t,b_t]$. Since $h$ is increasing in $I_0$, $[h\sp{-1}(a_0),h\sp{-1}(b_0)]$ is a subinterval of $I_0$. Thus, we may assume that $h[I_0]=I_1$. In particular, $h\!\!\restriction_{I_0}:I_0\to I_1$ is an order isomorphism.

Assume that $i=0$, the other case is entirely symmetric. Since $I_0\subset J_k\sp0$, by $(\ast)$, there are $p,q\in I_0\cap D$ such that $[p,q]$ is a jump in $\arrow$. Thus, $[h(p),h(q)]$ is a jump in $\arrow$ such that $h(p),h(q)\in I_1\cap D$. But then $h(p),h(q)\in J_m\sp1\cap D$ and $\pi(h(p))=\pi(h(q))$. This contradicts $(\ast)$ so we have proved the claim.

Thus, given that $D_\alpha\cap X=C_\alpha$ for all $\alpha<\C$, we obtain that $h[C_\beta]=C_\gamma$ from the claim. But this contradicts the choice of the family $\{C_\alpha:\alpha<\C\}$. Thus, there can be no such homeomorphism $h$ and we have finished the proof.
\end{proof}

\section{Other properties of $\arrow$}

There are some interesting properties of $\arrow$ that the author discovered while trying to prove the results in this paper. These results are not directly related to the main topic of the paper. As suggested by the referee, we include the proofs of these results but encourage the reader to treat them as exercises.

\begin{propo}
Let $C\subset\arrow_0$ be compact and scattered. Then there is a homeomorphism $h:\arrow\to\arrow$ such that $h[\arrow_0\setminus C]=\arrow_0$.
\end{propo}
\begin{proof}
We will prove this result by induction on the Cantor-Bendixson rank of $C$. By the compactness of $C$, such rank must be equal to $\tau+1$ for some ordinal $\tau$. The last non-empty Cantor-Bendixson derivative of $C$ is the derivative of order $\tau$, we will call this set $C\sp{(\tau)}$. 

There are pairwise disjoint clopen intervals $J_0,\ldots, J_m$ of $\arrow$ such that $|C\sp{(\tau)}\cap J_i|\leq 1$ for each $i\leq m$ and $\arrow=J_0\cup\ldots\cup J_m$. Every clopen interval of $\arrow$ is of the form $[\pair{p,0},\pair{q,1}]$ where $0\leq p<q\leq 1$. So $J_i$ is order-isomorphic (in particular, homeomorphic) to $\arrow$ for $i\leq m$. Thus, it is enough to define the homeomorphism in each of the intervals $J_i$ for $i\leq m$. For the rest of the proof let us assume without loss of generality that $C\sp{(\tau)}=\{\pair{p,0}\}$ for some $p\in(0,1]$.

First assume that $p\neq 1$, let $A_0=[\pair{0,1},\pair{p,0}]$ and $A_1=[\pair{p,1},\pair{1,0}]$. Let $f_0:[0,p]\to[1/2,1]$ and $f_1:[p,1]\to[0,1/2]$ be order isomorphisms and define $f:\arrow\to\arrow$ by 
$$
f(\pair{q,t})=\left\{
\begin{array}{cc}
\pair{f_0(q),t} & \textrm{ if }\pair{q,t}\in A_0,\\
\pair{f_1(q),t} & \textrm{ if }\pair{q,t}\in A_1.
\end{array}
\right.
$$
Then $f$ is a homeomorphism, $f[C]\subset\arrow_0$ and $f(\pair{p,0})=\pair{1,0}$. This shows that we may assume that $p=1$ for the rest of the proof.

Let $\{x_n:n<\omega\}\subset(0,1)$ be increasing such that $\sup\{x_n:n<\omega\}=1$. Let $I_0=[\pair{0,1},\pair{x_0,0}]$ and $I_{n+1}=[\pair{x_n,1},\pair{x_{n+1},0}]$ for $n<\omega$. Notice that 
$$
\arrow=\big(\bigcup\{I_n:n<\omega\}\big)\cup\{\pair{1,0}\}
$$
and $I_n$ is a clopen subset order-isomorphic to $\arrow$ for each $n<\omega$. Since $C\sp{(\tau)}=\{\pair{1,0}\}$, the Cantor-Bendixson rank of $C\cap I_n$ is at most $\tau$ for each $n<\omega$. Thus, by the inductive hypothesis, there exists a homeomorphism $h_n:I_n\to[\pair{1/(n+2),1},\pair{1/(n+1),0}]$ such that $h_n[(\arrow_0\cap I_n)\setminus (C\cap I_n)]=(1/(n+2),1/(n+1)]\times\{0\}$ for each $n<\omega$. We define $h:\arrow_0\to\arrow_0$ by
$$
h=\{\pair{\pair{1,0},\pair{0,1}}\}\cup\big(\bigcup\{h_n:n<\omega\}\big).
$$
It is not hard to see that $h$ is a homeomorphism and $h[\arrow_0\setminus C]=\arrow_0$. This completes the proof.
\end{proof}

\begin{propo}\label{arrowtimessequence}
$\arrow\times(\omega+1)$ is not homeomorphic to $\arrow$.
\end{propo}
\begin{proof}
Assume that there is a homeomorphism $h:\arrow\times(\omega+1)\to\arrow$, we will reach a contradiction. It is not hard to prove that every clopen subset of $\arrow$ is the finite union of clopen intervals. Then for all $n<\omega$ there is $m_n<\omega$ and a partition $\{J(n,i): i\leq m_n\}$ of $\arrow$ into clopen subsets such that $h[J(n,i)\times\{n\}]$ is a clopen interval for all $i\leq m_n$. The family
$$
\U=\{J(n,i):n<\omega,i\leq m_n\}
$$
is a countable collection of clopen subsets of $\arrow$. Since $\arrow$ has weight $\C$, $\U$ is not a base of $\arrow$. From the compactness of $\arrow$ it is possible to find $x,y\in\arrow$ with $x\neq y$ and such that $x\in W$ if and only if $y\in W$ for all $W\in\U$.

By the definition of the product topology, $\{\pair{x,n}:n<\omega\}$ converges to $\pair{x,\omega}$ and $\{\pair{y,n}:n<\omega\}$ converges to $\pair{y,\omega}$. Let $p=h(\pair{x,\omega})$, we will show that $\{h(\pair{y,n}):n<\omega\}$ converges to $p$, this contradicts the fact that $h$ is injective and we will have finished.

Assume that $p\in\arrow_0$, the other case is entirely analogous. It is enough to prove that for every $q\in\arrow$ with $q<p$ there is $k<\omega$ with $h(\pair{y,k})\in(q,p]$. Since $h[\arrow\times\{\omega\}]$ is crowded, there is $r\in h[\arrow\times\{\omega\}]\cap(q,p)$. By the continuity of $h$, there exists $k<\omega$ such that $h(\pair{x,k})\in(r,p)$. Let $j\leq m_k$ be such that $x\in J(k,j)$. Since $h[J(k,j)\times\{k\}]$ is an interval that intersects $(r,p)$ and does not contain its endpoints, $h[J(k,j)\times\{k\}]\subset(r,p)$. Then $h(\pair{y,k})\in (r,p)\subset(q,p]$. This proves that $\{h(\pair{y,n}):n<\omega\}$ converges to $p$ and as discussed above, finishes the proof. 
\end{proof}

\begin{coro}\label{cororeferee}
$\arrow$ does not contain any subspace homeomorphic to $\arrow\times(\omega+1)$.
\end{coro}
\begin{proof}
Apply Propositions \ref{closedinarrow} and \ref{arrowtimessequence}.
\end{proof}

So we can in fact prove the following result.

\begin{coro}
If $C$ is compact and countably infinite, then $C\times\arrow$ is neither homogeneous nor CDH.
\end{coro}
\begin{proof}
Since $C$ has a dense set of isolated points $D$, it is not hard to see that the set 
$$
W=\{x\in C\times\arrow:x\textrm{ has a clopen neighborhood homeomorphic to }\arrow\}
$$
is equal to the open dense subset $D\times\arrow$ of $C\times\arrow$. However, if $x\in C$ is a limit of a non-trivial sequence and $t\in\arrow$, every neighborhood of $\pair{x,t}$ contains a set homeomorphic to $\arrow\times(\omega+1)$. So there is no homeomorphism that takes $\pair{x,t}$ to $W$ by Corollary \ref{cororeferee}.

To show that $C\times\arrow$ is not CDH, modify the proof of \cite[Theorem 3.2]{arh-vm-cdh-cardinality} (or perhaps, the proof of our Theorem \ref{arrowtypes}) using that $W$ is invariant under homeomorphisms.
\end{proof}

The following result can be shown by modifying the proof of Proposition \ref{homeomorphisms} given in \cite{arh-vm-cdh-cardinality}. We will not give the details.

\begin{propo}\label{generalization}
Let $f:\arrow\to\arrow$ be a continuous function. Then there exists a collection $\U$ of pairwise disjoint clopen intervals of $\arrow$ such that $\bigcup\U$ is dense in $\arrow$ and for every $J\in\U$, $f\restriction_{J}:J\to\arrow$ is either non-increasing or non-decreasing.
\end{propo}

\begin{propo}
If $1\leq n<\omega$, then ${}\sp{n}\arrow$ is not CDH.
\end{propo}
\begin{proof}
Assume that $n\geq 2$, since the case $n=1$ was treated in \cite{arh-vm-cdh-cardinality}. Recall that $\Q$ denotes the set of rational numbers. Let $Q=\Q\cap(0,1)$, $D=Q\times\{0\}$ and $E=Q\times\{0,1\}$. Then $D$ and $E$ are countable dense subsets of $\arrow$ so ${}\sp{n}D$ and ${}\sp{n}E$ are countable dense subsets of ${}\sp{n}\arrow$. Assume that there is a homeomorphism $h:{}\sp{n}\arrow\to{}\sp{n}\arrow$ such that $h[{}\sp{n}E]={}\sp{n}D$, we will reach a contradiction.

Let $d\in D$ and define $X=\{x\in{}\sp{n}\arrow:x(i)=d\textrm{ for all }1\leq i\leq n-1\}$. We will identify $X$ with $\arrow$ in the obvious way. Also notice that ${}\sp{n}E\cap X$ is dense in $X$. For $i\leq n-1$, let $\pi_i:{}\sp{n}{\arrow}\to\arrow$ be the projection to the $i$-th coordinate and let $f_i=\pi_i\circ h\!\!\restriction_X:X\to\arrow$, this is a continuous function.

By Proposition \ref{generalization} there is some clopen interval $J\subset X$ such that $f_i\!\!\restriction_J:J\to\arrow$ is either non-decreasing or non-increasing for each $i\leq n-1$. Notice that it is impossible that $f_i$ is constant on an interval for all $i\leq n-1$ because this would contradict the injectivity of $h$. Furthermore, it is not hard to see that given a clopen interval $I\subset\arrow$ and a non-decreasing (or non-increasing) continuous function $f:I\to\arrow$, if $f$ is not constant on any clopen interval then $f$ is injective. Thus, we may assume that $f_j\!\!\restriction_J:J\to\arrow$ is one-to-one for some $j\leq n$. Also assume that $f_j\!\!\restriction_J$ is strictly increasing, the other case is similar.

Let $p,q\in J$ be such that $\pi(p)=\pi(q)\in Q$ and $q$ is the immediate successor of $p$ in the order of $X$, notice that $p,q\in{}\sp{n}E$ . By the fact that $f_j$ is strictly increasing, it is not hard to prove that $f_j(p)\in\arrow_0$ and $f_j(q)\in\arrow_1$. But $h(q)\in{}\sp{n}D$ by the choice of $h$ and $\pi_j[{}\sp{n}D]\subset D$ is disjoint from $\arrow_1$. This is a contradiction which shows that such a homeomorphism $h$ cannot exist. 
\end{proof}

\section*{acknowledgements}

The author would like to thank Professor Michael Hru\v s\'ak for suggesting the problems and for the fruitful discussions that took place. The referee's suggestions on the organization of the paper are also greatly appreciated.

\end{document}